\documentclass[12pt,letterpaper]{article}
\usepackage{amsmath}
\usepackage{amsfonts}
\usepackage{amsthm}
\usepackage{amssymb}
\usepackage{color}

\usepackage{verbatim} 
\usepackage{hyperref}
\usepackage{breakurl}
\usepackage{appendix} %% у меня он почему-то не хотел работать
\usepackage[utf8]{inputenc} %% но bibTeX вроде не дружит с utf8
\usepackage[pdftex]{graphicx}  %% для правильной графики в pdf
\usepackage{enumerate}

\newtheorem{prop}{Proposition}
\newtheorem{corollary}[prop]{Corollary}
\newtheorem{theorem}[prop]{Theorem}
\newtheorem{lemma}[prop]{Lemma}

\theoremstyle{remark}

\newcommand{\be}{\begin{equation}}
\newcommand{\ee}{\end{equation}}
\newcommand{\bes}{\begin{equation*}}
\newcommand{\ees}{\end{equation*}}
\newcommand{\bea}{\begin{eqnarray}}
\newcommand{\eea}{\end{eqnarray}}
\newcommand{\beas}{\begin{eqnarray*}}
\newcommand{\eeas}{\end{eqnarray*}}

\DeclareMathOperator{\FC}{FC}

\begin{document}

\title{Chameleon Coins}

\author{Tanya Khovanova \and Konstantin Knop \and Oleg Polubasov}

\maketitle

\begin{abstract}
We discuss coin-weighing problems with a new type of coin: a
chameleon. A chameleon coin can mimic a fake or a real coin, and it can choose
which coin to mimic for each weighing independently. 

We consider a mix of $N$ coins that include exactly two non-real
coins: one fake and one chameleon. The task is to use a balance
to find two coins one of which has to be fake. We find bounds for the
number of coins for which we can find a solution in a given number of
weighings. We also introduce an important idea of solution scaling.
\end{abstract}

\section{Introduction}

We all have played with problems where we had real coins and fake (counterfeit) coins. The second author invented a new type of a coin: \textit{a chameleon coin}. This coin can mimic a fake or a real coin. It also can choose which coin to mimic for each weighing independently. 

In coin-weighing literature many authors prefer using the word \textit{counterfeit}, rather than \textit{fake}. Unfortunately, the word ``counterfeit'' starts with the same letter as the word ``chameleon.'' Thus, we prefer to use the word ``fake'' and we utilize the letter F to denote the \textbf{f}ake coin. The letter C is reserved for the \textbf{c}hameleon coin.

You cannot find the chameleon coin if it does not want to be found, because it can consistently behave as either real or fake. 

Suppose we have real, fake and chameleon coins in the mix. The usual task of identifying the fake coins using a balance scale cannot be achieved: the chame\-le\-ons can pretend to be fake coins. What we can do is to find a small number of coins some of which are guaranteed to be fake.

Consider the simplest setup, when we have one fake coin and one chameleon in our mix of $N$ coins. The fake coin is lighter than real coins. All real coins weigh the same. Our task now is: to find \textit{two} coins, one of which has to be fake. We will call the finding of these two coins using the minimum number of weighings an \textit{FC-problem}. We denote $\FC(N)$ the smallest number of weighings for which the FC-problem with $N$ coins has a solution. The standard research goal would be to find the lower and upper bounds for $\FC(N)$.

It is often more convenient to solve the inverse problem: for a given number of weighings $w$ find or bound the largest number of coins $N(w)$ for which the problem has a solution.

The task of solving an $N$-coin problem in $w$ weighings we call a $(w,N)$-\textit{problem}. If the problem is solvable, we call its solution a $(w,N)$-\textit{solution} or a $(w,N)$-\textit{algorithm}.

We start with the main result, describing the bounds for $N(w)$ in Section~\ref{sec:main}. In Section~\ref{sec:st} we discuss the connection between our problem and the problem of finding two fake coins. 

In Section~\ref{sec:pmt} we produce the proof of the main result. We start the proof with an educational example of how to solve the problem for $3^n$ coins in $2n$ weighings in Section~\ref{sec:firstbound}. We continue by describing what happens for a small number of weighings in Section~\ref{sec:smalltotal}. Section~\ref{sec:scaling} discusses an important idea of scaling. We find bounds for scalable solutions in Section~\ref{sec:scalablebounds}. We introduce the notion of pseudo-solution in Section~\ref{sec:ps}. The ideas of scaling and pseudo-solutions allow us to reach our bounds and finish the main proof.

We follow with an important idea of invariants and monovariants for our algorithms in Section~\ref{sec:invariants}. Finally we discuss the bounds for $\FC(N)$ in Section~\ref{sec:any}.

\section{Main Result}\label{sec:main}

Our main result is the following theorem:

\begin{theorem}\label{thm:main}
The largest number of coins $N(w)$ for which there exists a $(w,N)$-algorithm and where the number of weighings does not exceed 5 is the following:
\begin{enumerate}
\item $N(0) = 2$
\item $N(2) = 4$
\item $N(3) = 6$
\item $N(4) = 11$
\item $N(5) = 20$.
\end{enumerate}
For larger values of $w$ the following bound holds:
\begin{enumerate}
\item $N(6+2k) \geq 36\cdot 3^k$, $k \geq 0$
\item $N(5+2k) \geq 20\cdot 3^k$, $k \geq1$.
\end{enumerate}
\end{theorem}

The first part of the theorem is due to our computer program that checked all possibilities. The bound for more than five weighings is proven in Section~\ref{sec:pmt}.

In Section~\ref{sec:st} we study the connection of our problem with the FF-problem. In the \textit{FF-problem} we are given that $N$ coins have exactly two fake coins in the mix. The fake coins weigh the same and are lighter than real coins. The goal is to find the fake coins. We use the connection between the FF and FC-problems to produce the information-theoretic bound (ITB) that serves as the upper bound for $N(w)$:

\begin{theorem}
\label{thm:bounds}
$N(w)\leq \lfloor \sqrt{2\cdot 3^w + 1/4} + 1/2\rfloor $.
\end{theorem}

Our results for up to ten weighings are presented in Table~\ref{tbl:itbound}. The second row shows the largest number of coins for which a solution with the given number of weighings is found. The numbers in bold show the proven best solution.

\begin{table}[h!]
  \begin{center}
\begin{tabular}{| r | r | r | r | r | r | r | r | r | r | r |}
  \hline                       
  number of weighings $w$ & 1 & 2 & 3 & 4 & 5 & 6 & 7 & 8 & 9 & 10\\
  $(w,n)$-algorithm found  & \textbf{2} & \textbf{4} & \textbf{6} & \textbf{11} & \textbf{20} & 36 & 60 & 108 & 180 & 324\\
  ITB for $N(w)$ & 3 & 4 & 7 & 13 & 22 & 38 & 66 & 115 & 198 & 344\\
  \hline  
\end{tabular}
  \end{center}
  \caption{Algorithm results and an ITB for a small number of weighings}
\label{tbl:itbound}
\end{table}

\section{The ITB and the FF-problem}\label{sec:st}

Now we prove the information-theoretic bound for $N(w)$ from Theorem~\ref{thm:bounds}.

\begin{proof}[Proof of Theorem ~\ref{thm:bounds}]
Suppose the chameleon always pretends to be fake. That means at the end we will find the fake coin and the chameleon. We need to differentiate between $\binom{N}{2}$ possibilities. Therefore, the smallest number of weighings supplied by the ITB must be at least $\log_3 \binom{N}{2}$. Or, the number of coins $N(w)$ that can be processed in $w$ weighings must be not more than$\lfloor \sqrt{2\cdot 3^w + 1/4} + 1/2\rfloor $. 
\end{proof}

If the chameleon always pretends to be fake, the FC-problem becomes the FF-problem. It is known that the largest number of coins that can be processed in the FF-problem is very close to the ITB \cite{Knop, Li}. In fact, the second and the third authors wrote a program and found the solutions to FF-problems that coincide with the bound for up to 10 weighings inclusive \cite{KP}.

If, on the other hand, the chameleon decides not to behave as a fake every time, then as soon as we find the fake coin we do not need to look for the second coin, the chameleon. That means it might be possible to find the fake coin faster then the ITB. As we are looking for the number of weighings that guarantees finding the fake coin, then the problem of finding two fake coins gives us a bound.

\begin{lemma}
Any algorithm that solves the FC-problem can be used to solve the FF-problem.
\end{lemma}

\begin{proof}
At the end of an FC-algorithm we have one coin or two coins left. The output of one coin means that this coin is the only one that can be fake. In the FF-problem such an outcome is impossible. Therefore, the corresponding branches of the FC-algorithm will not be used.

Two coins at the end of the FC-algorithm could mean one of two possibilities:

\begin{enumerate}
\item The output coins are the fake and the chameleon.
\item The output coins are the fake and a real coin. 
\end{enumerate}

In the second case we do not know which of the coins is fake and which one is real. But in any case, the chameleon is not in the output. This means the chameleon pretended to be real at some point. Thus this variant in the FF-problem will not appear.

In the FC-algorithm, we never lose the fake coin. That means using the same process for the FF-problem we will not lose a fake coin. That means our output would always consist of two fake coins.
\end{proof}

The ITB can also be deduced as an immediate corollary of the above lemma: The number of coins, $N$, in an $(w,N)$-solution cannot be more than the largest possible number of coins in  an FF-algorithm with the same number of weighings. On the other hand, not every algorithm for the FF-problem can be used to solve the FC-problem. The difference is already seen when we restrict ourselves to one weighing. If we have one fake and one chameleon, we can not process more than 2 coins. If we have 2 fake coins, we can process three coins. The FC-problem is way more complicated than the FF-problem as it is not known how the chameleon decides to behave.

\section{The Proof of the Main Theorem}\label{sec:pmt}

Our goal is to process the largest number of coins with a given number of weighings. The main idea behind our algorithms is scaling. Sometimes given a ($w,N$)-algorithm, it is possible to construct a ($w+2$,$3N$)-algorithm. The method that we use is called \textit{scaling} and ($w$,$N$)-algorithms/solutions that allow us to do that we call \textit{scalable}.  

In the next section we present a ($2n$,$3^n$)-solution to get us started.

\subsection{$3^n$ coins in $2n$ weighings}\label{sec:firstbound}

To describe a ($2n$,$3^n$)-solution let us start with 2 weighings and 3 coins. 

We use this example to explain our pseudo-code. Each line begins with its number. After it we have the weighing in the format 1 10 v 4 5 meaning coins 1 and 10 are weighed versus coins 4 and 5. The weighing is followed by a colon, after which we describe in order actions for three different results: equality, the first pan is lighter, and the second pan is lighter. Each action is one of the following:

\begin{itemize}
\item $\Rightarrow L$ means go to line $L$.
\item $(a)$ means only coin $a$ is fake.
\item $(a,b)$ means the fake coin is either $a$ or $b$.
\item $()$ means this branch is impossible and there is no output.
\item \textit{sym} indicates the symmetry of the weighing and its result; therefore the resulting go-to line is omitted as being equivalent to another line.
\end{itemize}

The line numbers after $\Rightarrow $ in line $L$ are always $3L+1$, $3L+2$ and $3L+3$. The \textit{sym} symbol implies that line $3L+3$ is omitted as a symmetric version of line $3L+2$.

This is a (2,3)-solution in our pseudo-code:

\hrulefill

\noindent\texttt{\indent First weighing:\\
0. 1 v 2 : $\Rightarrow$ 1, $\Rightarrow$ 2, sym.\\
\indent Second weighing:\\
1. 1 v 3 : (2,3), (1,2), (3).\\
2. 1 v 3 : (1,3), (1), (3).}

\hrulefill

In case the pseudo-code is difficult to follow, here is a detailed explanation. Denote the 3 coins $x_1$, $x_2$, and $x_3$. First, we compare coins $x_1$ and $x_2$. If the scale balances or $x_1$ is lighter, we compare coins $x_1$ and $x_3$. The case when $x_1$ is heavier is symmetric to the case when it is lighter and we will omit it from the following list explaining the outcomes of two weighings:

\begin{enumerate}
\item $x_1 = x_2$ and $x_1 = x_3$. Coin $x_1$ must be the chameleon. The fake coin is either $x_2$ or $x_3$.
\item $x_1 = x_2$ and $x_1 < x_3$. Coin $x_3$ cannot be fake. The fake coin is one of $x_1$ and $x_2$.
\item $x_1 = x_2$ and $x_1 > x_3$. Coin $x_1$ cannot be fake. If coin $x_2$ is fake, then $x_1$ has to be the chameleon and $x_3$ has to be real. This creates a contradiction. The fake coin is $x_3$.
\item $x_1 < x_2$ and $x_1 = x_3$. Similarly to case (ii), the fake coin is one of $x_1$ and $x_3$.
\item $x_1 < x_2$ and $x_1 < x_3$. Coin $x_1$ must be fake.
\item $x_1 < x_2$ and $x_1 > x_3$. Coin $x_1$ must be the chameleon, and coin $x_3$ must be fake.
\end{enumerate}

Now let us go back to $3^n$ coins. We start with dividing coins into three parts of equal sizes: $X_1$, $X_2$, and $X_3$. In the first weighing we compare $X_1$ with $X_2$. The case when $X_1 > X_2$ is omitted below as it can be resolved by symmetry. In the second weighing we compare $X_1$ against $X_3$.

Therefore, we have the following 6 cases for the first two weighings. To understand the future idea of scaling the reader may compare these cases to the 6 cases above when the total number of coins is 3:

\begin{enumerate}[(1)]
\item $X_1 = X_2$ and $X_1 = X_3$. Pile $X_1$ contains the chameleon. Piles $X_2$ or $X_3$ contain the fake coin.
\item $X_1 = X_2$ and $X_1 < X_3$. Piles $X_1$ and $X_2$ each contain a non-real coin, but we do not know which pile contains which coin.
\item $X_1 = X_2$ and $X_1 > X_3$. Pile $X_1$ cannot contain the fake coin. If $X_2$ contains the fake, then $X_1$ has to contain the chameleon and $X_3$ has to contain only real coins. This creates a contradiction. The fake coin is in pile $X_3$, which may or may not contain the chameleon.
\item $X_1 < X_2$ and $X_1 = X_3$. Similarly to case (ii), piles $X_1$ and $X_3$ each contain a non-real coin, but we do not know which pile contains which coin.
\item $X_1 < X_2$ and $X_1 < X_3$. Pile $X_1$ contains the fake coin. It may or may not contain the chameleon.
\item $X_1 < X_2$ and $X_1 > X_3$. Pile $X_3$ contains the fake coin and pile $X_1$ contains the chameleon.
\end{enumerate}

The result can be summarized as one of these three groups:

\begin{itemize}
\item Cases 1 and 6: We have a pile of size not more than $2 \cdot 3^{n-1}$ that contains the fake coin and does not contain the chameleon. 
\item Cases 2 and 4: We have two piles of size $3^{n-1}$, one containing the fake coin and the other the chameleon, but we do not know which is which. 
\item Cases 3 and 5: We have a pile of size $3^{n-1}$ that contains the fake coin and may or may not contain the chameleon.
\end{itemize}

Now the next step. 

In Cases 1 and 6, we have a pile that contains the fake coin and does not contain the chameleon. That means we are in the setting of the first-ever coin puzzle \cite{Dyson, Guy, Schell} of finding a single light fake coin out of $N$ coins. The standard method allows us to find one fake coin out of $3^{n-1}$ coins in $n-1$ weighings. Using the same method we can find two coins with one fake in $n-1$ weighings if the total number of coins does not exceed $2\cdot 3^{n-1}$. So in Cases 1 and 6 the total of $n+1$ weighings is enough to solve the problem.

In Cases 2 and 4, we have two piles each containing a non-real coin, but we do not know which pile contains which coin. We can process each pile separately as if it contains a false coin. After $n-1$ weighings for each pile we will end up with the fake coin and another coin which might or might not be the chameleon. The total number of weighings is $2n$.

In Cases 3 and 5, after two weighings we are at the same place as where we started but we have a pile 3 times smaller. Invoking induction we can find the two coins in $2n$ weighings.

The number of coins we can process with this algorithm for $2n$ weighings is $3^n$. We will later find faster algorithms.

\subsection{A small number of weighings}\label{sec:smalltotal}

Now we want to exhaustively discuss what happens for a small number of weighings.

Let us consider one weighing. If it is unbalanced, then the fake coin cannot be in the heavier pile. If it balances, the fake coin can be anywhere. Thus one weighing cannot solve the problem for any number of coins exceeding 2.

How many coins can we process in 2 weighings? We already know from Section~\ref{sec:firstbound} that we can process 3 coins. Can we do better? Yes, we can. Below we present our (2,4)-solution using the pseudo-code we described in Section~\ref{sec:firstbound}.

\hrulefill

\noindent\texttt{\indent First weighing:\\
0. 1 v 2 : $\Rightarrow$ 1, $\Rightarrow$ 2, sym.\\
\indent Second weighing:\\
1. 1 2  v  3 4 : (3,4), (1,2), (3,4).\\
2. 3 v 4 : (1), (1,3), (1,4).}

\hrulefill

In case the pseudo-code is difficult to follow, here is a detailed explanation. Denote the 4 coins $x_1$, $x_2$, $x_3$, and $x_4$. First, we compare coins $x_1$ and $x_2$. 

Suppose the first weighing balances: $x_1=x_2$. That means, if one of these two coins is fake, then the other has to be the chameleon. Then we compare $\{x_1,x_2\}$ against $\{x_3,x_4\}$. If the second weighing is unbalanced, then the lighter pan has the fake coin. If it is balanced, then the fake coin has to be on the opposite pan from the chameleon. That means $x_1$ and $x_2$ cannot both be non-real coins, which means the fake coin is one of $x_3$ and $x_4$.

Suppose the first weighing is not balanced: $x_1 < x_2$. That means $x_2$ cannot be fake and $x_1$ is either the fake or the chameleon. In the second weighing we  compare $x_3$ against $x_4$. If the second weighing balances, the fake coin cannot be there. Indeed, the fake coin can only balance against the chameleon, but the set $\{x_3,x_4\}$ cannot have both of them. If the second weighing is not balanced, then the heavier coin cannot be the fake one. Thus the two lighter coins contain the fake coin.

Thus we can process 4 coins in two weighings. Notice that this is more coins than we could process using the algorithm in Section~\ref{sec:firstbound}.

We performed an exhaustive computer search for the FC-problem. We found that the greatest number of coins that can be processed in two weighings is 4. Other computational results are in Table~\ref{tbl:smallvalues}. Starting from 6 weighings the computer was not powerful enough to completely answer the question. The last line shows the best known result for finding two fake coins \cite{KP}. The numbers in bold show the proven best solutions.

\begin{table}[h!]
  \begin{center}
\begin{tabular}{| r | r | r | r | r | r | r | r | }
  \hline                       
  number of weighings & 1 & 2 & 3 & 4 & 5 & 6 & 7\\
  \hline
  fake and chameleon & \textbf{2} & \textbf{4} & \textbf{6} & \textbf{11} & \textbf{20} & 36 & 62 \\
  two fake & \textbf{3} & \textbf{4} & \textbf{7} & \textbf{13} & \textbf{22} & \textbf{38} & \textbf{66}\\
  \hline  
\end{tabular}
  \end{center}
  \caption{Best known solutions for a small number of weighings}
\label{tbl:smallvalues}
\end{table}

Here we also show a (3,6)-solution. 

\hrulefill

\noindent\texttt{\indent First weighing:\\
0. 1 2 v 3 4 : $\Rightarrow$ 1, $\Rightarrow$ 2, sym.\\
\indent Second weighing:\\
1. 1 3 v 5 6 : $\Rightarrow$ 4, $\Rightarrow$ 5, (5, 6).\\
2. 1 v 2 : $\Rightarrow$ 7, $\Rightarrow$ 8,  sym.\\
\indent Third weighing:\\
4. 2 4 v 5 6 : (), (2, 4), (5, 6).\\
5. 1 4 v 2 3 : (1, 3), (1, 4), (2, 3).\\
7. 1 2 v 5 6 : (5, 6), (1, 2), (5, 6).\\
8. 5 v 6 : (1), (1, 5), (1, 6).}

\hrulefill

Note that the () in the output in line 4 means that this situation is impossible.

Line numbers are not consecutive because we skip line 3 as symmetric to line 2 and line 6 as unneeded (output (5, 6) was written at line 1). 

A (4,11)-solution, a (5-20)-solution, and a (6-36)-solution are presented in Appendix~\ref{app:4-11}, Appendix~\ref{app:5-20} and Appendix~\ref{app:6-36} respectively. 

Can we use the ideas we found for small number of weighings and extend them to a larger number of weighings? Our goal is to find the largest number of coins that can be processed with a given number of weighings. The main method here is scaling, and it will be covered in the next section.

\subsection{Scaling}\label{sec:scaling}

Let us go back to the classical problem of finding one fake coin, that is known to be lighter, from a set of coins. If a solution with $w$ weighings and $N$ coins exists, we can extend it to a solution with $w+1$ weighings and $3N$ coins. We do it by replacing every coin in the set of $N$ coins by three coins and perform $w$ weighings on these $N$ groups. At the end we know which group of three coins contains the fake coin and we can find it in one weighing. We want to extend this idea to our situation, but it is not that straightforward.

Now let us describe scaling. Consider a set of weighings for $N$ coins. Imagine that each coin represents a triple of coins. We can use our set of $w$ weighings on $3N$ coins by treating each triple as one coin. This set of weighings is called \textit{scaling}.

We call a ($w$,$N$)-solution \textit{scalable} if scaling the first $w$ weighings of this solution to $3N$ coins can be extended to a ($w+2$,$3N$)-solution. Not every solution is scalable.

Let us consider two examples.

First example: a (2,3)-solution from the beginning of Section~\ref{sec:firstbound} is scalable. The resulting (4,9)-solution is exactly the one described in Section~\ref{sec:firstbound}. If we continue scaling we get the ($2n$,$3^n$)-solution from the same section.

Second example: a (2,4)-solution is not scalable. If it were, then a (4,12)-solution would have existed, but our exhaustive search showed that it does not exist.

Let us study the (2,4)-solution from Section~\ref{sec:smalltotal} and see why exactly it is not scalable. Suppose we try to scale the (2,4)-solution once. We have 12 coins that are divided into four groups of three: $X_1$, $X_2$, $X_3$, and $X_4$. Consider the scaling of the two weighings of the (2,4)-solution and see what we can conclude. The  cases up to symmetry are:

\begin{itemize}
\item $X_1 = X_2$ and $X_1 + X_2 = X_3+ X_4$. Group $X_3 + X_4$ contains the fake coin. It may not contain the chameleon. We \textbf{can} finish in two weighings.
\item $X_1 = X_2$ and $X_1 + X_2 < X_3+ X_4$. Groups $X_1$ and $X_2$ contain one non-real coin each. We \textbf{can} finish in two weighings.
\item $X_1 = X_2$ and $X_1 + X_2 > X_3+ X_4$. Group $X_3 + X_4$ contains the fake coin. It may or may not contain the chameleon. We \textbf{cannot} finish in two weighings because the FC-problem cannot be solved for 6 coins in two weighings.

\item $X_1 < X_2$ and $X_3=X_4$. Group $X_1$ contains the fake coin. It may or may not contain the chameleon. We \textbf{can} finish in two weighings.
\item $X_1 < X_2$ and $X_3 < X_4$. Groups $X_1$ and $X_3$ contain one non-real coin each. We \textbf{can} finish in two weighings.
\end{itemize}

If some solutions are scalable and others are not, why are we interested in the scaling idea? The beauty is that it is easy to say which solution is scalable and which is not.

After these examples we can formulate when scaling \textbf{does not work}. Suppose at the end of the ($w$,$N$)-solution we found two coins $a$ and $b$ one of which is fake. These coins become 6 coins after the scaling and we know in what situations we cannot finish the solution in two weighings:

\begin{lemma}
We cannot scale that algorithm that outputs two coins $a$ and $b$ if every weighing is one of the following:
\begin{enumerate}
\item $a$ and $b$ are in the same (lighter) pan.
\item $a$ and $b$ are not on the scale and the scale balances.
\item $a$ is on the lighter pan and $b$ is not on the scale.
\item $b$ is on the lighter pan and $a$ is not on the scale.
\end{enumerate}
Otherwise, we can scale.
\end{lemma}

\begin{proof}
First we show that we cannot scale in the given cases. After scaling we get 6 coins: one of them is fake and the chameleon may be there. That means if the chameleon pretends to be fake, we will have to distinguish 15 pairs of coins, while in two weighings the best we can do is distinguish nine possibilities.

Now we want to show that in all other cases we can scale. 

Suppose we know that coins $a$ and $b$ were on the scale at some point or another, but not always in the same pan.

Suppose there was a weighing when they were opposite each other. The weighing has to be balanced. Otherwise the heavier pan does not contain the fake coin, thus no coin from the heavier pan can be in the output of the algorithm. After the scaling each group corresponding to $a$ and $b$ contains not more than one non-real coin, which we can find in one weighing per group. Thus we can scale in this case.

Now suppose that $a$ and $b$ were never opposite each other on the scale. None of the coins was ever on the heavier pan as otherwise it would not have been in the output.

\begin{enumerate}
\item If both $a$ and $b$ were on lighter pans, then after the scaling the corresponding groups have exactly one non-real coin. 
\item Suppose both coins $a$ and $b$ only participated in balanced weighings. That means all weighings are balanced and the fake coin was always opposite the chameleon. That means neither $a$ or $b$ can be a chameleon, but one of them is fake and we do not know which one. After scaling one of the groups $a$ and $b$ contains the fake coin and the other group contains only real coins. (Note. The existence of such situation might seem counter-intuitive, so here is a (3-4)-solution where this happens. We compare coins 1 against 2, 3 against 4, and 2 against 4. If all the weighings balance, then either 1 or 3 are fake.)
\item The case when $a$ always participates in a balanced weighing and $b$ is on a lighter pan at least once is impossible, as in this case $a$ cannot be fake.
\end{enumerate}

We showed that all other cases are scalable.
\end{proof}

Notice that in the (2,4)-solution, the coins 3 and 4 are always together and they are the output in some of the cases. Therefore, the (2,4)-solution is not scalable.

\begin{corollary}\label{thm:scalabilitycondition}
We can scale if each of the output coins $a$ and $b$ was on the scale at some point, and they were not always on the same pan.
\end{corollary}

The beauty of scaling is that if we can scale once we can scale many times.

\begin{theorem}
A scaling of a scalable ($w$,$N$)-solution produces a scalable ($w+2$,$3N$)-solution.
\end{theorem}

\begin{proof}
By Corollary~\ref{thm:scalabilitycondition} the output of a ($w$,$N$)-solution are coins $a$ and $b$ that were on the scale at some point, and not always on the same pan. After scaling, these two coins become two groups of three coins and the new output is two coins, one from each group. That means that the output of a ($w+2$,$3N$)-solution is two coins that were on the scale at some point and not always on the same pan. Therefore, the ($w+2$,$3N$)-solution is scalable.
\end{proof}

Our computer search found a scalable (4,10)-solution, which is shown in Appendix~\ref{app:4-10}. After scaling three times it becomes a (10,270)-solution, which is better than the (10,243)-solution we found in Section~\ref{sec:firstbound}.

Notice that the (3,6)-solution in Section~\ref{sec:smalltotal} is not scalable. Line 7 outputs coins 5 and 6 that were always together. 

Here we show a scalable (3,6)-solution:

\hrulefill

\noindent\texttt{\indent First weighing:\\
0. 1 2  v  3 4 : $\Rightarrow$ 1, $\Rightarrow$ 2, $\Rightarrow$ 3.  sym\\
\indent Second weighing:\\
1. 1  v  2 : $\Rightarrow$ 4, $\Rightarrow$ 5, $\Rightarrow$ 6.  sym \\
2. 5  v  6 : $\Rightarrow$ 7, $\Rightarrow$ 8, $\Rightarrow$ 9.  sym  \\
\indent Third weighing:\\
4. 3 5  v  4 6 : (5, 6), (3, 5), (4, 6). \\
5. 2 3  v  5 6 : (1, 4), (1, 3), (5, 6). \\
7. 1  v  2 : (1, 2), (1), (2). \\ 
8. 1  v  2 : (5), (1, 5), (2, 5). }

\hrulefill

Suppose there is a non-scalable solution. When scaling we might get six coins at the end that require three rather than two weighings to process. Interestingly, the scaled solution becomes itself scalable because there exist a scalable (3,6)-solution:

\begin{lemma}
Scaling a non-scalable ($w$,$N$)-solution generates a scalable ($w+3$,$3N$)-solution.
\end{lemma}

\subsection{Bounds for scalable weighings}\label{sec:scalablebounds}

In this section we present an upper bound for the number of coins that can participate in a scalable ($w$,$N$)-solution. 

If a ($w$,$N$)-solution is scalable, then there exists a ($w+2k$,$N\cdot 3^k$)-solution, for any positive $k$. Therefore, $\binom{N \cdot 3^k}{2} \leq 3^{w+2k}$. Equivalently, 
$$N \cdot 3^k \leq \lfloor \sqrt{2 \cdot 3^{w+2k}+1/4} + 1/2\rfloor.$$ 

Taking the limit when $k$ tends to infinity, we get:

$$N \leq \lfloor \sqrt{2 \cdot 3^{w}}\rfloor.$$

This bound either matches the ITB for any solution or is less than the ITB by 1. For example, from this bound we can see that a scalable (4,13)-solution does not exist without performing a computer search. We call this bound an \textit{induced bound from scalability}. 

In addition to that, we have information-theoretic consideration to suggest a slightly stronger bound:

\begin{lemma}\label{thm:ubducedscalablebound}
For a scalable ($w$,$N$)-solution $N(N+1)/2 \leq 3^w$, or 
$N\leq \lfloor \sqrt{2\cdot 3^w + 1/4} - 1/2\rfloor $.
\end{lemma}

\begin{proof}
Suppose we have $N$ coins that later will become groups of coins. If the chameleon and the fake coin are in different groups, we need to find both of them as the chameleon can pretend to be fake. If the chameleon and the fake coin are in the same group we just need to find the group. Overall we have to produce $N(N+1)/2$ answers.
\end{proof}

The ITB for scalability differs from the ITB for any solution $N(w)$exactly by 1. The results are summarized in Table~\ref{tbl:itboundscalable}.

\begin{table}[h!]
  \begin{center}
\begin{tabular}{| r | r | r | r | r | r | r | r | r | r | r |}
  \hline                       
  number of weighings $w$ & 1 & 2 & 3 & 4 & 5 & 6 & 7 & 8 & 9 & 10\\
  ITB bound for $N(w)$ & 3 & 4 & 7 & 13 & 22 & 38 & 66 & 115 & 198 & 344\\
induced bound for scalability & 2 & 4 & 7 & 12 & 22 & 38 & 66 & 114 & 198 & 343\\
ITB for scalability & 2 & 3 & 6 & 12 & 21 & 37 & 65 & 114 & 197 & 343\\
  \hline  
\end{tabular}
  \end{center}
  \caption{Bounds for scalability}
\label{tbl:itboundscalable}
\end{table}

For example, the ITB for scalalbility allows us to deduce, without performing a computer search, that the (3,7)-scalable solution does not exist.

\subsection{Pseudo-solutions}\label{sec:ps}

According to our exhaustive computer search a scalable (4,11)-solution does not exist.  But we found a non-solution that is scalable. How can a non-solution become a solution?

Consider an example. Suppose you have an algorithm that outputs $k$ coins such that the fake is there and the chameleon is not. If we replace each coin by three coins and add two more weighings, then we can find two coins containing one fake out of total of 18 coins. Thus, if $k < 7$, the non-solution after scaling becomes a solution. We call such a non-solution a \textit{pseudo-solution}. 

Here is a (4,11)-pseudo-solution:

\hrulefill

\noindent\texttt{\indent First weighing:\\
0. 1 2 3 4  v  5 6 7 8 : $\Rightarrow$ 1, $\Rightarrow$ 2, $\Rightarrow$ 3.  sym\\
\indent Second weighing:\\
1. 5 6 1 2 3  v  7 4 9 10 11 : $\Rightarrow$ 4, $\Rightarrow$ 5, $\Rightarrow$ 6.\\
2. 1 9  v  2 10 : $\Rightarrow$ 7, $\Rightarrow$ 8, $\Rightarrow$ 9.  sym\\
\indent Third weighing:\\
4. 5 1  v  6 2 : $\Rightarrow$ 13, $\Rightarrow$ 14, $\Rightarrow$ 15.  sym\\
5. 1  v  2 : $\Rightarrow$ 16, $\Rightarrow$ 17, $\Rightarrow$ 18.  sym\\
6. 8 4  v  9 1 : $\Rightarrow$ 19, $\Rightarrow$ 20, $\Rightarrow$ 21.\\
7. 1 10 11  v  3 4 5 : $\Rightarrow$ 22, $\Rightarrow$ 23, $\Rightarrow$ 24.\\
8. 9  v  11 : $\Rightarrow$ 25, $\Rightarrow$ 26, (1,11).\\
\indent Fourth weighing:\\
13. 7  v  8 : \{3,4,9,10,11\}, (3,7), (8).\\
14. 7  v  4 : \{1,5,8,9,10,11\}, (1,7), (4,5).\\
16. 5  v  6 : (3,8), (3,5), (3,6).\\
17. 5  v  6 : (1,8), (1,5), (1,6).\\
19. 7 10  v  9 11 : (10,11), (7,10), (9,11).\\
20. 1 7  v  10 11 : (4,8), (4,7), (10,11).\\
21. 10  v  11 : (7,9), (9,10), (9,11).\\
22. 2 9  v  3 5 : (4,11), (2,9), (3,11).\\
23. 10  v  11 : (1,2), (1,10), (11).\\
24. 3  v  4 : (3,4), (3), (4).\\
25. 3  v  4 : (1), (1,3), (1,4).\\
26. 3  v  4 : (1,9), (3,9), (4,9).}

\hrulefill

Lines 13 and 14 correspond to the lack of solution. They produce a list of 5 or 6 coins that do not contain the chameleon, but contain the fake.

After scaling, this pseudo-solution becomes a scalable (6,33)-solution. Indeed, these 5(6) coins after the scaling become 15(18) coins, that do not contain the chameleon. In two weighings we can reduce this group to 2 coins as required. After more scalings we get a ($4+2k$,$11 \cdot 3^k$) solution.

We proved computationally that a (4,12)-pseudo-solution deso not exist.

We also found a scalable (6,36)-solution presented in Appendix~\ref{app:6-36}, which propagates to a ($6+2n$,$36\cdot 3^n$)-solution, providing our bound for an even number of weighings. In addition, we found a scalable (5,20)-solution presented in Appendix~\ref{app:5-20}, which propagates to a ($5+2n$,$20\cdot 3^n$)-solution, providing a bound for an odd number of weighings.

This completes the proof of Theorem~\ref{thm:main}.

\section{Invariants}\label{sec:invariants}

In this section we want to discuss some invariants and monovariants that allowed us to speed up the program and find scalable solutions and pseudo-solutions.

To do it we want to represent $N$ coins that after scaling will become $N$ groups of three coins as a graph with $N$ vertices. A directed edge from $a$ to $b$ means if the group corresponding to $a$ contains a fake, then the group corresponding to $b$ may contain a chameleon. A loop from $a$ to $a$ means after the scaling, the group corresponding to $a$ may contain both non-real coins. We will call two edges from $a$ to $b$ and from $b$ to $a$ a \textit{double edge}.

The starting graph, before the weighings, has edges from every vertex to every vertex, including a loop at every vertex. The starting graph for three vertices is in Figure~\ref{fig:starting}.

\begin{figure}[htbp]
\centering
\includegraphics[scale=0.3]{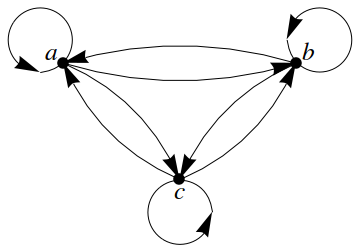}
\caption{The starting graph for three vertices}\label{fig:starting}
\end{figure}

We consider the following quantities:

\begin{itemize}
\item $D$ is the number of double edges. Before the weighings $D=N(N-1)/2$.
\item $E$ is the number of vertices that have outgoing edges leading to vertices without outgoing edges. Before the weighings $E=0$. We will also use $E$ to refer to those vertices.
\item $F$ is the number of loops. Before the weighings $F=N$.
\end{itemize}

After each weighing we replace the graph with three new graphs, corresponding to the three different outcomes. In Figure~\ref{fig:afterfirstweighing} we show what happens with the starting graph above after the first weighing comparing $a$ and $b$. The graph on the left represents the outcome $a=b$. If the fake coin is in the group corresponding to $a$, then the chameleon can only be in $b$, and vice versa. If the fake coin is in the group corresponding to $c$, then the chameleon can be anywhere. The graph in the middle represents the outcome $a<b$. If the fake coin is in the group corresponding to $a$, the chameleon can be anywhere. The fake coin cannot be in $b$. If the fake coin is in $c$, then the chameleon is in $a$. The graph on the right represents the outcome $a>b$ and is symmetric to the graph in the middle.

\begin{figure}[htbp]
\centering
$\begin{array}{ccc}
\includegraphics[scale=0.25]{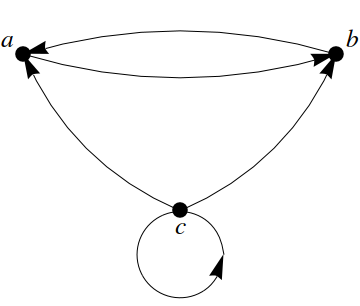} &
\includegraphics[scale=0.3]{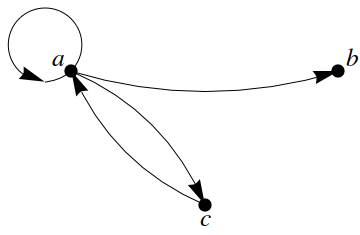} &
\includegraphics[scale=0.3]{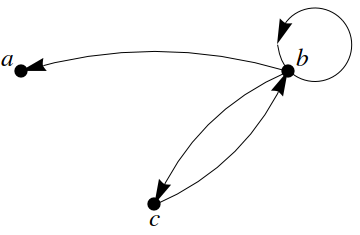}
\end{array}$
\caption{The graph after the first weighing}\label{fig:afterfirstweighing}
\end{figure}

We sum the values for all of the outcomes. Notice that in the example above, after the first weighing $D=3$, $E=2$, and $F=3$.

We first show that $D$ and $F$ are invariant after any number of weighings.

\begin{lemma}
The value $D$ is invariant.
\end{lemma}

\begin{proof}
Consider a weighing of a group of coins $A$ against a group $B$. Denote by $C$ the leftover pile.

If the weighing balances, then either the fake coin is in one group and the chameleon is in the other, or the fake coin is in group $C$ and the chameleon can be anywhere. That means all the double edges between vertices inside $A$ and between vertices$B$, as well as the double edges from $C$ to $A$ or $B$ disappear. Equivalently, only the double edges between $A$ and $B$ and inside $C$ remain.

If $A$ is lighter than $B$, then $B$ cannot contain the fake coin, and $C$ cannot contain both non-real coins. Moreover, if $C$ contains the fake coin, then the chameleon is in $A$. Thus, all the outgoing edges from $B$ and all the directed edges inside $C$ disappear. Also, the directed edges from $C$ to $B$ disappear. Equivalently, only the double edges inside $A$ and the double edges from $A$ to $C$ remain.

By similarity when $A$ is lighter than $B$, then only the double edges inside $B$ and the double edges from $B$ to $C$ remain.

To summarize, each double edge remains for exactly one outcome of the weighing.
\end{proof}

\begin{lemma}
The value $F$ is invariant.
\end{lemma}

\begin{proof}
Consider a weighing of a group of coins $A$ against a group $B$. Denote by $C$ the leftover pile.

If the weighing balances then either the fake coin is in one group and the chameleon is in the other, or the fake coin is in group $C$ and the chameleon can be anywhere. That means the loops for all the coins in both groups $A$ and $B$ disappear. Equivalently, only loops in $C$ remain.

If $A$ is lighter than $B$, then $B$ cannot contain the fake coin, and $C$ cannot contain both non-real coins. Moreover, if $C$ contains the fake coin, then the chameleon is in $A$. That means all the loops in $B$ and $C$ disappear. Equivalently, only loops in $A$ remain.

Similarly, if $A$ is heavier than $B$, only loops in $B$ remain.

To summarize, each loop remains for exactly one outcome of the weighing.
\end{proof}

Now we want to discuss the value $E$.

\begin{lemma}
The value $E$ does not decrease after a weighing.
\end{lemma}

\begin{proof}
What happens to the vertices from $E$? If a coin from $E$ participates in a weighing,  it stays in $E$ if and only if it is on the lighter pan. 

If a coin from $E$ does not participate in a weighing and the weighing balances, then the coin stays in the set. If the weighing does not balance, the coin may or may not  disappear from $E$. In any case, a coin from $E$ remains in $E$ in at least one of the three new graphs corresponding to different outcomes of the weighing.

Also, new vertices in $E$ can appear. Therefore, the value of $E$ is non-decreasing.
\end{proof}

What happens to the graphs and the values $D$, $E$, and $F$ at the end of a scalable solution or pseudo-solution?

At the end of a scalable solution (pseudo-solution) each output has to be one of the following:

\begin{enumerate}
\item Two coins with one double edge and no loops. 
\item Up to 2 coins from $E$ for a solution or upto 6 coins from $E$ for a pseudo-solution.
\item A coin with a loop, accompanied by not more than one coin from $E$.
\end{enumerate}

\begin{theorem}\label{thm:scalablebound}
At the end of a scalable ($w$,$N$)-pseudo-solution we have: 
$$D + |E-F|/6 + F \le 3^w.$$
\end{theorem}

\begin{proof}
The number of answers of type (i) is $D$. The number of answers of type (iii) is $F$. The number of answers of type (ii) is up to $(E-F)/6$. This is because some of the coins from $E$ could have participated in the type (iii) answer, but no more than $F$ of those. The rest of the coins from $E$ are divided into groups of not more than 6.
\end{proof}

At the end of a scalable ($w$,$N$)-pseudo-solution we have: 
$$D+ F \le 3^w.$$
Plugging in $D=N(N-1)/2$ and $F=N$, we get the corollary:

\begin{corollary}
$$N(N+1) \le 3^w.$$
\end{corollary}

For a scalable solution this corollary provides another proof of Lemma~\ref{thm:ubducedscalablebound}. We now see that the same bound is true for a pseudo-solution.

Theorem~\ref{thm:scalablebound} is a very useful tool in trimming the useless branches during the run of our program.

\section{Any Number of Coins}\label{sec:any}

Up to now our goal was to find and/or bound the largest number of coins we can process with a given number of weighings: $N(w)$. What about the original question of finding $\FC(N)$? Can we bound the number of weighings for any number of coins? 

Our exhaustive computer search combined with bounds provides the information about $\FC(N)$ for small $N$ that is collected in Table~\ref{tbl:FC(N)}.

\begin{table}[!h]
  \begin{center}
\begin{tabular}{| r | r | r | r | r | r | r | r | r | r | r |}
  \hline                       
  $N$      & 2 & 3-4 & 5-6 & 7-11 & 12-20    &   21-36 & 37-38 & 39-62\\
  $\FC(N)$ & 1 & 2 &      3 &      4             & 5   & 6  & 6 or 7             & 7\\
  \hline  
\end{tabular}
  \end{center}
  \caption{$\FC(N)$ for a small number of coins}
\label{tbl:FC(N)}
\end{table}

It is possible to extend the algorithms we know to process any number of coins that require not too many weighings. Suppose $N =aK +r$, where $0 \leq r < a$. The following theorem estimates the bound.

\begin{theorem}
$\FC(N) \leq \FC(K) + \FC(2a+r)$.
\end{theorem}

\begin{proof}
Use the best algorithm for $K$ coins on $K$ groups each containing $a$ coins. The algorithm will output two groups of size $a$. The fake coin is either in these two groups or in the leftover group of $r$ coins. Use the best algorithm for these $2a+r$ coins. It has to output the fake coin.
\end{proof}

\begin{corollary}
$\FC(N) \leq  \FC(\lfloor N/a \rfloor) + \FC(3a-1)$.
\end{corollary}

It follows that $\FC(N) \leq  \FC(\lfloor N/3 \rfloor)  + 4$, and $\FC(N) \leq  \FC(\lfloor N/5 \rfloor)  + 5$, and $\FC(N) \leq \FC(\lfloor N/9 \rfloor) + 6$.

Our discussion in Section~\ref{sec:st} supplied us with the following lower bound for $\FC(N)$:
$$\FC(N) \geq \log_3 \binom{N}{2} \approx \log_3 2 \cdot \log_3 N.$$

The corollary above gives us the upper bound: 
$$\FC(N) \leq 3 \log_3 N.$$

Our intuition and experiments suggest that in reality $\FC(N)$ is close to the lower bound.

\section{Acknowledgements}

We are grateful to Julie Sussman, P.P.A., for editing this paper.

\begin{appendices}

\section{A Scalable (4-10)-Solution}\label{app:4-10}
\noindent\texttt{\indent First weighing:\\
0. 1 2 3  v  4 5 6 : $\Rightarrow$ 1, $\Rightarrow$ 2, $\Rightarrow$ 3. sym\\
\indent Second weighing:\\
1. 1 7  v  2 8 : $\Rightarrow$ 4, $\Rightarrow$ 5, $\Rightarrow$ 6. sym\\
2. 7 8  v  9 10 : $\Rightarrow$ 7, $\Rightarrow$ 8, $\Rightarrow$ 9. sym\\
\indent Third weighing:\\
4. 4 9  v  5 10 : $\Rightarrow$ 13, $\Rightarrow$ 14, $\Rightarrow$ 15. sym\\
5. 4 9  v  5 10 : $\Rightarrow$ 16, $\Rightarrow$ 17, $\Rightarrow$ 18. sym\\
7. 1  v  2 : $\Rightarrow$ 22, (1, 3), $\Rightarrow$ 24. sym\\
8. 7  v  8 : $\Rightarrow$ 25, $\Rightarrow$ 26, $\Rightarrow$ 27. sym\\
\indent Fourth weighing:\\
13. 8 7  v  10 9 : (3,6), (7,8), (9,10).\\
14. 3  v  9 : (4,9), (3,4), (9).\\
16. 1  v  7 : (6,7), (1,6), (7).\\
17. 4 1  v  7 9 : (7,9), (1,4), (7,9).\\
22. 1  v  3 : (2,3), (1,2), (3).\\
25. 1  v  2 : (3), (1), (2).\\
26. 1  v  2 : (3,7), (1,7), (2,7)}.

\section{A Non-scalable (4,11)-Solution}\label{app:4-11}

\noindent\texttt{\indent First weighing:\\
0. 1 2 3  v  4 5 6 : $\Rightarrow$ 1, $\Rightarrow$ 2, $\Rightarrow$ 3.  sym\\
\indent Second weighing:\\
1. 1 2 7 8  v  3 9 10 11 : $\Rightarrow$ 4, $\Rightarrow$ 5, $\Rightarrow$ 6. \\
2. 7 8  v  9 10 : $\Rightarrow$ 7, $\Rightarrow$ 8, $\Rightarrow$ 9.  sym  \\
\indent Third weighing:\\
4. 7  v  8 : $\Rightarrow$ 13, $\Rightarrow$ 14, $\Rightarrow$ 15.  sym \\
5. 3 4 1  v  5 7 8 : $\Rightarrow$ 16, $\Rightarrow$ 17, $\Rightarrow$ 18. \\
6. 4 9 10  v  5 11 1 : $\Rightarrow$ 19, $\Rightarrow$ 20, $\Rightarrow$ 21. \\
7. 1  v  2 : $\Rightarrow$ 22, $\Rightarrow$ 23, $\Rightarrow$ 24.  sym \\
8. 7  v  8 : $\Rightarrow$ 25, $\Rightarrow$ 26, $\Rightarrow$ 27.  sym \\
\indent Fourth weighing:\\
13. 4 9  v  5 10 : (6, 11), (4, 9), (5, 10).\\
14. 9  v  10 : (7, 11), (7, 9), (7, 10).\\   
16. 6 2  v  7 8 : (1, 5), (2, 6), (7, 8).\\
17. 6  v  2 : (1, 4), (1, 6), (2, 4). \\
18. 5 2  v  7 8 : (7, 8), (2, 5), (7, 8).\\ 
19. 6 3  v  9 2 : (10, 11), (3, 6), (9, 11). \\
20. 4 3  v  9 10 : (9, 10), (3, 4), (9, 10). \\
21. 3  v  11 : (5, 11), (3, 5), (11).\\ 
22. 2 1  v  3 11 : (3, 11), (1, 2), (3, 11). \\
23. 11  v  2 : (1, 3), (1, 11), (1). \\
25. 1  v  2 : (3), (1), (2).\\
26. 1  v  2 : (3, 7), (1, 7), (2, 7).}   

\section{A Scalable (5,20)-Solution}\label{app:5-20}
\noindent\texttt{\indent First weighing:\\
0. 1 2 3 4 5 6  v  7 8 9 10 11 12 : $\Rightarrow$ 1, $\Rightarrow$ 2, $\Rightarrow$ 3.  sym\\
\indent Second weighing:\\
1. 7 1 2 3  v  13 14 15 16 : $\Rightarrow$ 4, $\Rightarrow$ 5, $\Rightarrow$ 6.\\
2. 13 14 15 16  v  17 18 19 20 : $\Rightarrow$ 7, $\Rightarrow$ 8, $\Rightarrow$ 9.  sym\\
\indent Third weighing:\\
4. 4 13 14 17 18  v  5 15 16 19 20 : $\Rightarrow$ 13, $\Rightarrow$ 14, $\Rightarrow$ 15.  sym\\
5. 1 4  v  2 5 : $\Rightarrow$ 16, $\Rightarrow$ 17, $\Rightarrow$ 18.  sym\\
6. 13 17 18  v  14 19 20 : $\Rightarrow$ 19, $\Rightarrow$ 20, $\Rightarrow$ 21.  sym\\
7. 1 2  v  3 4 : $\Rightarrow$ 22, $\Rightarrow$ 23, $\Rightarrow$ 24.  sym\\
8. 1 2 3 13  v  4 5 6 14 : $\Rightarrow$ 25, $\Rightarrow$ 26, $\Rightarrow$ 27.  sym\\
\indent Fourth weighing:\\
13. 8 9 19  v  10 11 17 : $\Rightarrow$ 40, $\Rightarrow$ 41, $\Rightarrow$ 42.\\
14. 5 8 9 10  v  13 14 17 18 : $\Rightarrow$ 43, $\Rightarrow$ 44, $\Rightarrow$ 45.\\
16. 8 9 17 18  v  10 11 19 20 : $\Rightarrow$ 49, $\Rightarrow$ 50, $\Rightarrow$ 51.  sym\\
17. 8 9 17 18  v  10 11 19 20 : $\Rightarrow$ 52, $\Rightarrow$ 53, $\Rightarrow$ 54.  sym\\
19. 13 19 17  v  15 16 20 : $\Rightarrow$ 58, $\Rightarrow$ 59, $\Rightarrow$ 60.\\
20. 17  v  18 : $\Rightarrow$ 61, $\Rightarrow$ 62, $\Rightarrow$ 63.  sym\\
22. 1  v  2 : $\Rightarrow$ 67, $\Rightarrow$ 68, $\Rightarrow$ 69.  sym\\
23. 5  v  6 : $\Rightarrow$ 70, $\Rightarrow$ 71, $\Rightarrow$ 72.  sym\\
25. 4 1 2  v  5 15 16 : $\Rightarrow$ 76, $\Rightarrow$ 77, $\Rightarrow$ 78.\\
26. 15  v  16 : $\Rightarrow$ 79, $\Rightarrow$ 80, $\Rightarrow$ 81.  sym\\
\indent Fifth weighing:\\
40. 12 6  v  20 18 : (17, 19), (6, 12), (18, 20).\\
41. 7 8  v  19 18 : (6, 9), (6, 8), (18, 19).\\
42. 7 10  v  20 17 : (6, 11), (6, 10), (17, 20).\\
43. 1 11  v  17 18 : (4, 12), (4, 11), (17, 18).\\
44. 8  v  9 : (4, 10), (4, 8), (4, 9). \\
45. 13 17  v  14 18 : (17, 18), (13, 17), (14, 18). \\
49. 12  v  6 : (3, 7), (3, 12), (6, 7).\\
50. 4 8  v  17 18 : (3, 9), (3, 8), (17, 18).\\
52. 12  v  4 : (1, 7), (1, 12), (4, 7).\\
53. 2 8  v  17 18 : (1, 9), (1, 8), (17, 18).\\
58. 14 18  v  15 16 : (13, 20), (14, 18), (15, 16).\\
59. 19  v  17 : (13, 14), (13, 19), (14, 17).\\
60. 15 20  v  16 1 : (15, 16), (15, 20), (16).\\
61. 15  v  16 : (13), (13, 15), (13, 16). \\
62. 15  v  16 : (13, 17), (15, 17), (16, 17). \\
67. 3 5  v  4 6 : (5, 6), (3, 5), (4, 6). \\
68. 2 3  v  5 6 : (1, 4), (1, 3), (5, 6).\\
70. 1  v  2 : (1, 2), (1), (2). \\
71. 1  v  2 : (5), (1, 5), (2,5). \\
76. 3 14  v  15 16 : (6, 13), (3, 14), (15, 16).\\
77. 4 13  v  1 3 : (2, 14), (4, 13), (1, 14).\\
78. 15  v  16 : (5, 13), (15), (16). \\
79. 1  v  2 : (3, 13), (1, 13), (2,13). \\
80. 1  v  2 : (3, 15), (1, 15), (2,15).} 

\section{A Scalable (6,36)-solution}\label{app:6-36}
\noindent\texttt{\indent First weighing:\\
0. 1 2 3 4 5 6 7 8 9 10 11 12  v  13 14 15 16 17 18 19 20 21 22 23 24 : $\Rightarrow$ 1, $\Rightarrow$ 2, $\Rightarrow$ 3. sym \\
\indent Second weighing:\\
1. 13 14 15 1 2 3 4 25 26 27 28 29  v  5 6 7 8 9 30 31 32 33 34 35 36 : $\Rightarrow$ 4, $\Rightarrow$ 5, $\Rightarrow$ 6. \\
2. 1 2 3 25 26 27 28 29 30  v  4 5 6 31 32 33 34 35 36 : $\Rightarrow$ 7, $\Rightarrow$ 8, $\Rightarrow$ 9.  sym  \\
\indent Third weighing:\\ 
4. 13 16 17 25 26  v  14 18 19 27 28 : $\Rightarrow$ 13, $\Rightarrow$ 14, $\Rightarrow$ 15.  sym \\
5. 16 17 18 19 10 25 26  v  20 21 22 23 11 27 28 : $\Rightarrow$ 16, $\Rightarrow$ 17, $\Rightarrow$ 18.  sym \\
6. 16 17 5 6 30 31  v  7 32 33 34 35 36 : $\Rightarrow$ 19, $\Rightarrow$ 20, $\Rightarrow$ 21. \\
7. 1 2 3 31 32 33  v  7 8 9 10 11 12 : $\Rightarrow$ 22, $\Rightarrow$ 23, $\Rightarrow$ 24. \\
8. 25 26 27  v  28 29 30 : $\Rightarrow$ 25, $\Rightarrow$ 26, $\Rightarrow$ 27.  sym  \\
\indent Fourth weighing:\\ 
13. 20 5 6 10 30 31  v  21 7 8 11 32 33 : $\Rightarrow$ 40, $\Rightarrow$ 41, $\Rightarrow$ 42.  sym\\ 
14. 5 6 7 10 30 31 32 25  v  8 11 12 33 34 35 36 26 : $\Rightarrow$ 43, $\Rightarrow$ 44, $\Rightarrow$ 45. \\
16. 24 12 25 29  v  1 2 3 30 : $\Rightarrow$ 49, $\Rightarrow$ 50, $\Rightarrow$ 51. \\
17. 13 1 2 29  v  16 3 25 26 : $\Rightarrow$ 52, $\Rightarrow$ 53, $\Rightarrow$ 54. \\
19. 18 19 20 32 33 30  v  21 22 23 34 35 31 : $\Rightarrow$ 58, $\Rightarrow$ 59, $\Rightarrow$ 60.  sym \\
20. 16 18 19 30 31  v  20 21 22 5 1 : $\Rightarrow$ 61, $\Rightarrow$ 62, $\Rightarrow$ 63. \\
21. 18 19 32 33  v  20 7 34 1 : $\Rightarrow$ 64, $\Rightarrow$ 65, $\Rightarrow$ 66. \\
22. 25 26 27  v  28 29 30 : $\Rightarrow$ 67, $\Rightarrow$ 68, $\Rightarrow$ 69.  sym \\
23. 31 32 33  v  34 35 36 : $\Rightarrow$ 70, $\Rightarrow$ 71, $\Rightarrow$ 72. \\
24. 7 8 9  v  10 11 12 : $\Rightarrow$ 73, $\Rightarrow$ 74, $\Rightarrow$ 75.  sym \\
25. 7 8 9  v  10 11 12 : $\Rightarrow$ 76, $\Rightarrow$ 77, $\Rightarrow$ 78.  sym \\
26. 7 8 9  v  10 11 12 : $\Rightarrow$ 79, $\Rightarrow$ 80, $\Rightarrow$ 81.  sym  \\
\indent Fifth weighing:\\ 
40. 22 23 24 12  v  34 35 36 29 : $\Rightarrow$ 121, $\Rightarrow$ 122, $\Rightarrow$ 123. \\
41. 20 12 30 31 29  v  22 23 24 10 32 : $\Rightarrow$ 124, $\Rightarrow$ 125, $\Rightarrow$ 126. \\
43. 16 17 30 31  v  33 34 35 14 : $\Rightarrow$ 130, $\Rightarrow$ 131, $\Rightarrow$ 132. \\
44. 1 16 17 10  v  30 31 32 25 : $\Rightarrow$ 133, $\Rightarrow$ 134, $\Rightarrow$ 135. \\
45. 1 16 11  v  33 34 35 : $\Rightarrow$ 136, $\Rightarrow$ 137, $\Rightarrow$ 138. \\
49. 13 14 15 4  v  27 28 26 29 : $\Rightarrow$ 148, $\Rightarrow$ 149, $\Rightarrow$ 150. \\
50. 13 27 25  v  14 29 5 : $\Rightarrow$ 151, $\Rightarrow$ 152, $\Rightarrow$ 153. \\
51. 1  v  2 : $\Rightarrow$ 154, $\Rightarrow$ 155, $\Rightarrow$ 156.  sym \\
52. 14 25 26 29  v  17 18 19 4 : $\Rightarrow$ 157, $\Rightarrow$ 158, $\Rightarrow$ 159. \\
53. 17 10  v  18 29 : $\Rightarrow$ 160, $\Rightarrow$ 161, $\Rightarrow$ 162. \\
54. 17 18 4  v  25 26 5 : $\Rightarrow$ 163, $\Rightarrow$ 164, $\Rightarrow$ 165. \\
58. 24 8 9 32  v  34 35 36 30 : $\Rightarrow$ 175, $\Rightarrow$ 176, $\Rightarrow$ 177. \\
59. 18 32  v  19 33 : $\Rightarrow$ 178, $\Rightarrow$ 179, $\Rightarrow$ 180.  sym \\
61. 18 23 8  v  19 24 9 : $\Rightarrow$ 184, $\Rightarrow$ 185, $\Rightarrow$ 186.  sym \\
62. 18 8 9  v  30 31 1 : $\Rightarrow$ 187, $\Rightarrow$ 188, $\Rightarrow$ 189. \\
63. 20 23  v  21 24 : $\Rightarrow$ 190, $\Rightarrow$ 191, $\Rightarrow$ 192.  sym \\
64. 21 22 34 32 33  v  23 24 35 36 2 : $\Rightarrow$ 193, $\Rightarrow$ 194, $\Rightarrow$ 195. \\
65. 35  v  36 : $\Rightarrow$ 196, $\Rightarrow$ 197, $\Rightarrow$ 198.  sym \\
66. 21 22 23  v  34 35 36 : $\Rightarrow$ 199, $\Rightarrow$ 200, $\Rightarrow$ 201. \\
67. 7 8 34  v  9 10 35 : $\Rightarrow$ 202, $\Rightarrow$ 203, $\Rightarrow$ 204.  sym \\
68. 25  v  26 : $\Rightarrow$ 205, $\Rightarrow$ 206, $\Rightarrow$ 207.  sym \\
70. 1  v  2 : $\Rightarrow$ 211, $\Rightarrow$ 212, $\Rightarrow$ 213.  sym \\
71. 31  v  32 : $\Rightarrow$ 214, $\Rightarrow$ 215, $\Rightarrow$ 216.  sym \\
72. 34  v  35 : $\Rightarrow$ 217, $\Rightarrow$ 218, $\Rightarrow$ 219.  sym \\
73. 7  v  8 : $\Rightarrow$ 220, $\Rightarrow$ 221, $\Rightarrow$ 222.  sym \\
74. 7  v  8 : $\Rightarrow$ 223, $\Rightarrow$ 224, $\Rightarrow$ 225.  sym \\
76. 1  v  2 : $\Rightarrow$ 229, $\Rightarrow$ 230, $\Rightarrow$ 231.  sym \\
77. 7  v  8 : $\Rightarrow$ 232, $\Rightarrow$ 233, $\Rightarrow$ 234.  sym \\
79. 25  v  26 : $\Rightarrow$ 238, $\Rightarrow$ 239, $\Rightarrow$ 240.  sym \\
80. 25  v  26 : $\Rightarrow$ 241, $\Rightarrow$ 242, $\Rightarrow$ 243.  sym  \\
\indent Sixth weighing:\\ 
121. 21 10  v  20 11 : (9, 15), (10, 21), (11, 20). \\
122. 22  v  23 : (12, 24), (12, 22), (12, 23).  \\
123. 34  v  35 : (29, 36), (29, 34), (29, 35). \\
124. 20 10  v  5 1 : (6, 15), (10, 20), (5, 15).\\ 
125. 20 12  v  30 16 : (29, 31), (12, 20), (29, 30). \\
126. 22  v  23 : (10, 24), (10, 22), (10, 23). \\
130. 36 25  v  32 26 : (9, 13), (25, 36), (26, 32). \\
131. 16 17  v  30 13 : (26, 31), (16, 17), (26, 30). \\
132. 33  v  34 : (25, 35), (25, 33), (25, 34). \\
133. 5  v  6 : (7, 13), (5, 13), (6, 13). \\
134. 16  v  17 : (10), (10, 16), (10, 17). \\
135. 30  v  31 : (25, 32), (25, 30), (25, 31). \\
136. 17 12  v  36 26 : (8, 13), (12, 17), (26, 36). \\
137. 17  v  12 : (11, 16), (11, 17), (12, 16). \\
138. 33  v  34 : (26, 35), (26, 33), (26, 34). \\
148. 1  v  2 : (3, 24), (1, 24), (1, 24). \\
149. 13  v  14 : (4, 15), (4, 13), (4, 14). \\
150. 27  v  28 : (26, 29), (26, 27), (26, 28). \\
151. 24 4  v  28 29 : (12, 15), (4, 24), (28, 29). \\
152. 13 12  v  27 1 : (25, 28), (12, 13), (25, 27). \\
153. 12  v  29 : (14, 29), (12, 14), (29). \\
154. 13  v  14 : (3, 15), (3, 13), (3, 14). \\
155. 13  v  14 : (1, 15), (1, 13), (1, 13). \\
157. 15 10  v  1 3 : (2, 16), (10, 15), (1, 16). \\
158. 14 10  v  25 3 : (26, 29), (10, 14), (25, 29).\\ 
159. 17  v  18 : (4, 19), (4, 17), (4, 18). \\
160. 13 29  v  1 3 : (2, 19), (13, 29), (1, 19). \\
161. 13 10  v  1 3 : (2, 17), (10, 13), (1, 17). \\
162. 1  v  2 : (18, 29), (1, 18), (2, 29). \\
163. 5 19  v  25 26 : (3, 16), (3, 19), (25, 26). \\
164. 16 4  v  17 1 : (3, 18), (4, 16), (3, 17). \\
165. 25  v  26 : (25, 26), (25), (26).  \\
175. 1 16  v  33 31 : (7, 17), (7, 16), (31, 33). \\
176. 1 8  v  32 31 : (9, 24), (8, 24), (31, 32). \\
177. 34  v  35 : (30, 36), (30, 34), (30, 35). \\
178. 1 8  v  36 30 : (9, 20), (8, 20), (30, 36).\\ 
179. 1 8  v  32 30 : (9, 18), (8, 18), (30, 32). \\
184. 17 6  v  30 31 : (5, 16), (6, 17), (30, 31). \\
185. 17 8  v  23 6 : (5, 18), (8, 17), (6, 23). \\
187. 1 19  v  30 31 : (6, 16), (6, 19), (30, 31). \\
188. 18 6  v  8 1 : (9, 16), (6, 18), (8, 16). \\
189. 30  v  31 : (30, 31), (30), (31). \\
190. 17  v  6 : (5, 22), (5, 17), (6, 22). \\
191. 23  v  6 : (5, 20), (5, 23), (6, 20). \\
193. 1 18  v  35 36 : (7, 19), (7, 18), (35, 36). \\
194. 21 22  v  32 2 : (33, 34), (21, 22), (32, 34). \\
195. 23 35  v  24 36 : (35, 36), (23, 35), (24,36). \\
196. 18 32  v  19 33 : (32, 33), (18, 32), (19,33). \\
197. 32  v  33 : (35), (32, 35), (33, 35). \\
199. 24  v  34 : (7, 20), (7, 24), (34). \\
200. 21  v  22 : (7, 23), (7, 21), (7,22). \\ 
201. 35  v  36 : (34), (34, 35), (34, 36). \\
202. 4 36  v  11 12 : (5, 6), (4, 36), (11, 12). \\
203. 34  v  4 : (7, 8), (34), (). \\
205. 4  v  5 : (6, 27), (4, 27), (5, 27). \\
206. 4  v  5 : (6, 25), (4, 25), (5, 25). \\
211. 4  v  5 : (3, 6), (3, 4), (3, 5). \\
212. 4  v  5 : (1, 6), (1, 4), (1, 5). \\
214. 1  v  2 : (3, 33), (1, 33), (2, 33). \\
215. 1  v  2 : (3, 31), (1, 31), (2, 31). \\
217. 1  v  2 : (3, 36), (1, 36), (2, 36). \\
218. 1  v  2 : (3, 34), (1, 34), (2, 34). \\
220. 10  v  11 : (9, 12), (9, 10), (9, 11). \\
221. 10  v  11 : (7, 12), (7, 10), (7, 11). \\
223. 7  v  9 : (8, 9), (7, 8), (9). \\
224. 9  v  1 : (7), (7, 9), (7). \\
229. 1  v  3 : (2, 3), (1, 2), (3).\\ 
230. 3  v  2 : (1), (1, 3), (1). \\
232. 1  v  2 : (3, 9), (1, 9), (2, 9). \\
233. 1  v  2 : (3, 7), (1, 7), (2, 7). \\
238. 1  v  2 : (3, 27), (1, 27), (2, 27).\\
239. 1  v  2 : (3, 25), (1, 25), (2, 25). \\
241. 7  v  8 : (9, 27), (7, 27), (8, 27). \\
242. 7  v  8 : (9, 25), (7, 25), (8, 25). }

\end{appendices}

\end{document}